\documentclass[a4paper,12pt]{article}
\usepackage{amsmath,amssymb,amsthm,graphics,latexsym,amsfonts}
\usepackage{fancyhdr}
\usepackage{color}
\usepackage{graphics}

\title{\Large  Independent transversal domination number of a graph
\thanks{Research supported by NSFC (No. 11571294) }}
\author{ {Hongting Wang, Baoyindureng Wu \footnote{Corresponding author.
Email: baoywu@163.com (B. Wu) }, Xinhui An}\\
\small  College of Mathematics and System Sciences, Xinjiang
University \\ \small  Urumqi, Xinjiang 830046, P.R.China \\}
\date{}

\newtheorem{theorem}{Theorem}[section]
\newtheorem{conjecture}{Conjecture}[section]
\newtheorem{lemma}[theorem]{Lemma}
\newtheorem{corollary}[theorem]{Corollary}

\newtheorem{problem}[theorem]{Problem}

\usepackage{indentfirst}

\begin{document}
\maketitle {\small \noindent{\bfseries Abstract}:
Let $G=(V, E)$ be a graph. A set $S\subseteq V(G)$ is a
{\it dominating set} of $G$ if every vertex in $V\setminus S$ is adjacent to a vertex of $S$. The
{\it domination number} of $G$, denoted by $\gamma(G)$, is the cardinality of a minimum dominating set of $G$.
Furthermore, a dominating set $S$ is an {\it independent transversal dominating set} of $G$ if it intersects every maximum independent
set of $G$. The {\it independent transversal domination number} of $G$, denoted by $\gamma_{it}(G)$, is
the cardinality of a minimum independent
transversal dominating set of $G$.
In 2012, Hamid initiated the study of the independent transversal domination of graphs, and posed the following two conjectures:

Conjecture 1. If $G$ is a non-complete connected graph on $n$ vertices, then $\gamma_{it}(G)\leq\lceil\frac{n}{2}\rceil$.

Conjecture 2. If G is a connected bipartite graph, then $\gamma_{it}(G)$ is either $\gamma(G)$ or $\gamma(G)+1$.

We show that Conjecture 1 is not true in general. Very recently, Conjecture 2 is partially verified to be true by Ahangar, Samodivkin, Yero.
Here, we prove the full statement of Conjecture 2.
In addition, we give a correct version of
a theorem of Hamid. Finally, we answer a problem posed by Mart\'{i}nez, Almira, and Yero
on the independent transversal total domination of a graph.


\section {\large Introduction}
We consider undirected finite simple graphs only, and follow the notations and terminology in \cite{Bondy}.
Let $G=(V(G), E(G))$ be a graph. The {\it order} of $G$ is $|V(G)|$. For a vertex $v\in V(G)$, the
set of neighbors is denoted by $N(v)$. The degree of $v$, denoted by $d(v)$, is number of edges incident with $v$ in $G$.
Since $G$ is simple, $d(v)=|N(v)|$. The minimum degree of $G$, denoted by $\delta(G)$, is $\min\{d(v):\ v\in V(G)\}$.
For a set $S\subseteq V(G)$, let $N(S)=\cup_{v\in S} N(v)$.
A set $S$ is said to be an {\it independent set} of $G$, if no pair of vertices of $S$ are adjacent in $G$.
The {\it independence number} of $G$, denoted by $\alpha(G)$, is the cardinality of a maximum independent set of $G$.
We denote by $\Omega(G)$ the set of all maximum independent sets of $G$.

On the other hand, $S$ is said to be a {\it vertex
covering} of $G$ if every edge of $G$ is incident with some vertex of $S$ in $G$. It is easy to see that $S$ is an independent set of $G$ if and
only if $V(G)\setminus S$ is a vertex covering of $G$. The vertex covering number, denoted by $\beta(G)$, is the cardinality of a minimum vertex covering of $G$.
So, for any graph $G$ of order $n$,  \begin{equation}\alpha(G)+\beta(G)=n. \end{equation}

A set $M\subseteq E(G)$ is said to be a matching of $G$ if no pair of edges of $M$ have a common end vertex. The matching number of
$G$, denoted by $\alpha'(G)$, is the cardinality of a maximum matching of $G$. It is obvious that for a graph $G$ of order $n$, \begin{equation} \alpha'(G)\leq \min\{\frac n 2,\ \beta(G)\}.\end{equation}
The well-known K\"{o}nig-Egerv\'{a}ry theorem states that
for a bipartite graph $G$, \begin{equation}\alpha'(G)=\beta(G).\end{equation}

A set $S\subseteq V(G)$ is a
{\it dominating set} of $G$ if every vertex in $V\setminus S$ is adjacent to a vertex of $S$. The
{\it domination number} of $G$, denoted by $\gamma(G)$, is the minimum cardinality of a dominating set of $G$.
A minimum dominating set of a graph $G$ is called a $\gamma(G)$-set of $G$.

It is clear that for a graph $G$,  \begin{equation}\gamma(G)\leq \alpha(G), \end{equation} and if $G$ has no isolated vertices,
\begin{equation}\gamma(G)\leq \beta(G). \end{equation}

A set $S \subseteq V(G)$ is said to be an {\it independent transversal} of $G$ if $S\cap I\neq \emptyset$ for any $I\in \Omega(G)$.
The {\it independent transversal number} of $G$, denoted by $\tau_i(G)$, is the cardinality of a minimum independent transversal of
$G$.
A dominating set $S$ of a graph $G$ is said to be an {\it independent transversal dominating set}
if $S\cap I\neq \emptyset$ for any $I\in \Omega(G)$. By the definitions above, for any graph $G$,
\begin{equation}\max\{\gamma(G), \tau_i(G)\}\leq \gamma_{it}(G). \end{equation}

%

The notion of independent
transversal domination was first introduced by Hamid \cite{H} in 2012.
He proved that for a graph $G$ of order $n$, $\gamma_{it}(G)\leq n$, with equality
if and only if $G\cong K_n$.

\begin{theorem}(Hamid \cite{H})
For a graph $G$ without isolated vertices, $\gamma_{it}(G)\leq \beta(G)+1$.
\end{theorem}

\begin{theorem}  (Hamid \cite{H})
If $G$ is a non-complete connected graph of order $n$ with $\alpha(G)\geq \frac n 2$, then $\gamma_{it}(G)\leq \frac n 2$.
\end{theorem}

In view of the above theorem, Hamid posed the following conjecture in \cite{H}.

\begin{conjecture} (Hamid \cite{H})
If $G$ is a non-complete connected graph on $n$ vertices, then $\gamma_{it}(G)\leq\lceil\frac{n}{2}\rceil$.
\end{conjecture}

We show that the above conjecture is not true in general.
Hamid \cite{H} showed that $\gamma(G)\leq \gamma_{it}(G)\leq \gamma(G)+\delta(G)$. In particular, $\gamma_{it}(T)$
is either $\gamma(T)$ or $\gamma(T)+1$  for a tree $T$. So, he proposed the following conjecture.

\begin{conjecture} (Hamid \cite{H})
If G is a connected bipartite graph, then $\gamma_{it}(G)$ is either $\gamma(G)$ or $\gamma(G)+1$.
\end{conjecture}

Recently, Ahangar, Samodivkin, Yero \cite{Ah} proved that Conjecture 2 is valid for all unbalanced bipartite graphs.

\begin{theorem} (Ahangar, Samodivkin, Yero \cite{Ah})
Let $G$ be a bipartite graph with bipartition $(X, Y)$ such that $|X|\neq |Y|$. Then, $\gamma_{it}(G)\leq \gamma(G)+1$. In
particular, this is true when $G$ has odd order.

\end{theorem}

We show that the full statement of Conjecture 2 in the next section.

Complexity of independent transversal domination problem can be in \cite{Ah}.

\section {\large Disproof of Conjecture 1.1 and Proof of Conjecture 1.2}
First we begin with a useful observation.
\begin{lemma}
Let $G$ be non-complete graph of order $n$. If $G$ is the complement of a triangle-free graph, then $\tau_i(G)=\beta(\overline{G})$, and thus
$\gamma_{it}(G)\geq n-\alpha(\overline{G})$.
\end{lemma}

\begin{proof}
Since $G$ is the complement of a triangle-free graph, $\alpha(G)=2$.
Thus, $\Omega(G)=\{\{u, v\}:\ uv\in E(\overline{G})\}$ and
an independent transversal of $G$ is a vertex covering of $\overline{G}$, $\tau_i(G)=\beta(\overline{G})$, and thus
$\gamma_{it}(G)\geq n-\alpha(\overline{G})$.
\end{proof}


To disprove Conjecture 1.1, we recall a celebrated result on the Ramsey theory, due to Kim \cite{Kim}.

\begin{theorem} (Kim \cite{Kim}) For sufficiently large $n$, there exists a triangle-free graph
$G$ of order $n$ with $\alpha(G)\leq 9\sqrt{nlog\ n}$.
\end{theorem}

Note that the symbol $``\leq''$ in the inequality above was misprinted as $``\geq''$ in \cite{Kim}. The following
result is an immediate consequence of Lemma 2.1 and Theorem 2.2.
\begin{corollary}
For sufficiently large $n$, there exists a non-complete graph $G$ (the complement of a triangle-free graph) of order $n$,
$\gamma_{it}(G)\geq n-\alpha(\overline{G})\geq n-9\sqrt{nlog\ n}>\lceil \frac n 2\rceil$.
\end{corollary}

This disproves Conjecture 1. It is straightforward to check that
the complement of the Petersen graph $P_{10}$ is also a counterexample to Conjecture 1. Alon \cite{Al} gave
some explicit construction of triangle-free graphs with relatively small independence numbers contrast to
their orders. The complements of these graphs are also counterexamples to Conjecture 1.

\vspace{2mm} Let $core(G)=\cap\{S:S\in\Omega(G)\}$ be the set of vertices belonging to all maximum independent sets,
and let $\xi(G)=|core(G)|$.
\begin{theorem}  (Boros, Golumbic, and Levit \cite{BGL})  If $G$ is a connected graph with $\alpha(G)>\alpha'(G)$, then $\xi(G)\geq\alpha(G)-\alpha'(G)+1$.

\end{theorem}

It can be deduced from the theorem above that $\tau_i(G)=1$ for any connected graph $G$ with $\alpha(G)>\alpha'(G)$. It is an
interesting problem for characterizing graphs with $\tau_i(G)=1$.

\begin{problem}
What is the best upper bound of $\tau_i(G)$ for graphs $G$ in terms of their order $n$ ?
\end{problem}

\vspace{2mm} Now, we are ready to show Conjecture 2.

\begin{theorem} If G is a connected bipartite graph, then $\gamma_{it}(G)$ is either $\gamma(G)$ or $\gamma(G)+1$.
\end{theorem}

\begin{proof} Let $S$ be a $\gamma(G)$-set. If $ \alpha(G)>\alpha'(G)$, then by Theorem 2.4, $\xi(G)\geq\alpha(G)-\alpha'(G)+1>0$.
So, $S\cup \{v\}$ is an independent transversal dominating set for a vertex $v\in core(G)$, and hence $\gamma_{it}(G)\leq |S\cup \{v\}|=\gamma(G)+1$.

Next, we consider the case $\alpha(G)\leq \alpha'(G)$. Since $G$ is bipartite, \begin{equation}\alpha(G)\geq \frac n 2\geq \alpha'(G).\end{equation}
Combining (7) with (1), (2), (3), we have \begin{equation} \alpha(G)=\alpha'(G)=\beta(G)=\frac{n}{2}. \end{equation}

If $\gamma(G)=\beta(G)$, then the result follows from Theorem 1.1. So, we assume that $\gamma(G)<\beta(G)$.
Let $(X, Y)$ be the bipartition of $G$. By the equation (8), we have $|X|=|Y|=\frac n 2$. Thus, $S\cap X\neq \emptyset$ and
$S\cap Y\neq \emptyset$. Take a vertex $x\in S\cap X$, and let $\Omega_x$ be the set of all maximum independent set of $G$ containing $x$.
In particular, $X\in \Omega_x$. We consider $G-x$.
It is clear that $\Omega\setminus \Omega_x=\Omega(G-x)$ and $Y\in \Omega\setminus \Omega_x$.  Note that $$\alpha(G-x)=|Y|=\frac n 2> \frac n 2 -1=|X\setminus \{x\}|\geq \alpha'(G-x).$$
By Theorem 2.4, $\xi(G-x)>0$, and let $y\in core(G-x)$. So, $S\cup \{y\}$ is an independent transversal dominating set of $G$. This shows $\gamma_{it}(G)\leq \gamma(G)+1$.

\end{proof}

%


\section{\large Bipartite graphs $G$ with $\gamma_{it}(G)=\frac n 2$}

Hamid \cite{H} obtained the following theorem.
\begin{theorem}(Hamid \cite{H})
For a bipartite graph $G$ with bipartition $(X,Y)$ such that $|X|\leq |Y|$ and $\gamma(G)=|X|$, $\gamma_{it}(G)=\gamma(G)+1$ if and only if every vertex in $X$ is adjacent to at least two pendant vertices.
\end{theorem}

Actually, the theorem above is not complete, see for the counterexample in Figure 1. It is easy to see that
for the graph $G$, $\gamma(G)=2=|X|$ and $\gamma_{it}(G)=3$. However,
there is vertex in $X$, which has no two pendent neighbors.

\begin{center}
\scalebox{0.1}{\includegraphics{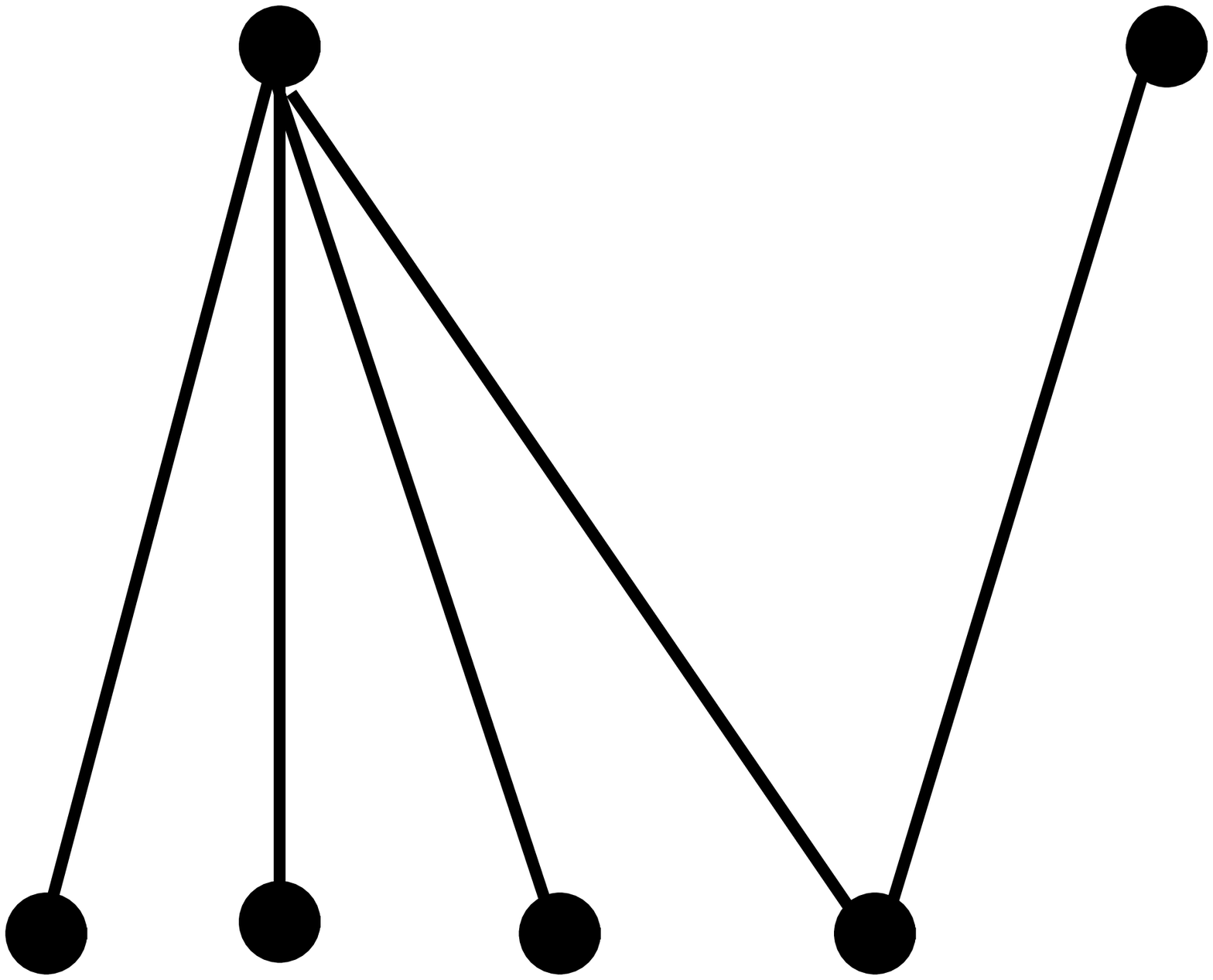}}\\
\vspace{0.4cm} Figure 1. A counterexample $G$
\end{center}

Next we present the correct version of the theorem.

\begin{theorem}  Let $G$ be a bipartite graph with bipartition $(X,Y)$ such that $|X|\leq |Y|$ and $\gamma(G)=|X|$, then $\gamma_{it}(G)=\gamma(G)+1$ if and only if every vertex in $X$ is either a pendent vertex, or is adjacent to at least two pendant vertices.
\end{theorem}

\begin{proof}

First, we show its necessity.
By the assumption that $|X|+1=\gamma(G)+1=\gamma_{it}(G)$ and by Theorem 1.1, we have $|X|\leq \beta(G)$. On the other
hand, since $X$ is a
vertex covering of $G$, $\beta(G)\leq |X|$. So, $X$ is a minimum covering of $G$, implying that $Y$ is
a maximum independent set of $G$.

Next, we show that $\delta(G)=1$. Suppose $\delta(G)\geq 2$. Since $G$ is bipartite and $\gamma(G)=|X|$,
$X$ is a $\gamma(G)$-set. Let $u\in X$ and $v\in N(u)$. Since $\delta(G)\geq 2$, it follows that $S=(X\setminus \{u\})\cup \{v\}$ is a
dominating set of $G$.

\vspace{2mm} \noindent{\bf Claim 1.} $S$ is an independent transversal dominating set of $G$.

Suppose it is not, and let $I$ be a maximum independent set of $G$ such that $I\cap S=\emptyset$.
Since $|I|=\alpha(G)=|Y|$ and $S=(X\setminus \{u\})\cup \{v\}$, we have
$I=(Y\setminus \{v\})\cup \{u\}$. However, since $d(u)\geq 2$, $u$ has a neighbor
in $I$. It contradicts the fact that $I$ is an independent set of $G$. This proves the
claim.

\vspace{2mm} By Claim 1, we have $\gamma_{it}(G)\leq |S|=|X|=\gamma(G)$, contradicting the assumption that $\gamma_{it}(G)=|X|+1$. This
proves $\delta(G)=1$.

\vspace{2mm} To complete the proof of the necessity, it suffices to show that if $u\in X$ is not a pendent vertex, then $u$ has at least two
pendent neighbors. Suppose, on the contrary, that $u$ has at most one pendent neighbor. Since $d(u)\geq 2$,
we choose a neighbor $v$ of $u$ with degree as small as possible.
Let $S=(X\setminus \{u\})\cup \{v\}$. We show that $S$ is a dominating set of $G$. If it is not,
there exists a vertex $y\in Y\setminus \{v\}$ not dominated by $S$. Then $y$ must be a pendent vertex and be adjacent to $u$.
Now, $u$ has two pendent neighbors $v$ and $y$, a contradiction. So, $S$ is a (minimum) dominating set of $G$.
Furthermore, by the similar argument as in the proof of Claim 1, we can show that $S$ is an independent transversal dominating set of $G$.
Again, we have $\gamma_{it}(G)\leq |S|=|X|=\gamma(G)$, contradicting the assumption that $\gamma_{it}(G)=\gamma(G)+1$. This
shows that $u$ has at least two
pendent neighbors.

%

\vspace{2mm} To show its sufficiency, assume that
every vertex in $X$ is either a pendent vertex, or is adjacent to at least two pendant vertices. Let $\{x_1, \ldots, x_k\}$ be the set of all pendent
vertices in $X$ for integer $k\geq 0$. Let $y_i$ be the neighbor of $x_i$ in $Y$ for each $i$. Then $\Omega(G)=\{I: I=(Y\setminus Y')\cup X'\}$,
where $X'\subseteq \{x_1, \ldots, x_k\}$ and $Y'=N(X')\cap \{y_1, \ldots, y_k\}$. In particular, if $X'=\emptyset$,
then $Y'=\emptyset$, and $I=Y$.

Note that if $S$ is a minimum dominating set of $G$, then there exists $X'\subseteq \{x_1, \ldots, x_k\}$ such
that $S=(X\setminus X')\cup Y'$, where $Y'$ is defined as above. However, $I_S=(Y\setminus Y')\cup X'$ is
a maximum independent set such that $S\cap I_S=\emptyset$.
It means that no minimum dominating set of $G$ is an independent transversal dominating set, implying that $\gamma_{it}(G)\geq \gamma(G)+1$.
On the other hand, by Theorem 2.1, $\gamma_{it}(G)\leq \beta(G)+1\leq |X|+1=\gamma(G)+1$. We conclude that $\gamma_{it}(G)=\gamma(G)+1$.

\vspace{2mm} The proof is completed.

\end{proof}

By Theorem 1.2, $\gamma_{it}(G)\leq \frac n 2$ for any bipartite graph $G$ without a component of order at most 2.
Hamid \cite{H} posed the following problem: characterizing all bipartite graphs $G$ for which $\gamma_{it}(G)=\frac{n}{2}$.
In what follows, we partially answer the problem.
If $G$ is a connected bipartite graph with $\gamma_{it}(G)=\frac{n}{2}$, then by Theorem 2.6,
\begin{equation}\frac n 2-1\leq \gamma(G)\leq \frac n 2.\end{equation}

The corona of a graph $G$, denoted by $G\circ K_1$, is the graph obtained from $G$ by joining a new vertex $v'$ to each vertex $v\in V(G)$.
In 1982, Payan and Xuong \cite{Pa} and independently, in 1985, Fink, Jacobson, Kinch and Roberts \cite{Fink} characterized the graphs $G$ of
order $n$ with $\gamma(G)=\frac n 2$.

\begin{theorem} (Fink et al. \cite{Fink}, Payan et al. \cite{Pa}) For any graph $G$ with even order $n$ having no isolated vertices
$\gamma(G)=\frac{n}{2}$ if and only if the components of $G$ are $C_{4}$ or $H\circ K_1$ for any connected graph $H$.
\end{theorem}
\begin{corollary} Let $G$ be a connected graph of even order $n\geq 4$. If $\gamma(G)=\frac n 2$, then
$\gamma_{it}(G)=\frac n 2$.
\end{corollary}
\begin{proof}
Since $\gamma(G)=\frac n 2$, by Theorem 3.3, either $G\cong C_4$ or $H\circ K_1$, where $H$ is a connected graph of order $\frac n 2\geq 2$.
It is straightforward to check that $\gamma_{it}(C_4)=2$. If $G\cong H\circ K_1$ for a connected graph of order $\frac n 2$, $\gamma_{it}(G)=\frac n 2$, because
the set of pendent vertices is the unique minimum independent transversal dominating set of $G$. So, the result follows.
\end{proof}

\begin{theorem}
For a bipartite graph $G$ of even order $n$ with bipartition $(X, Y)$ and without a component of order at most 2, $\gamma_{it}(G)=\frac n 2$, if
one of the following cases occurs:

(1) $\gamma(G)=\frac n 2$.

(2) $\gamma(G)=\frac n 2-1$, and $|X|=\frac n 2-1$, $|Y|=\frac n 2+1$, and every vertex in $X$ is either a pendent vertex, or is adjacent to at least two pendant vertices.


\end{theorem}

\begin{proof}

If $\gamma(G)=\frac n 2$, by Corollary 3.4, $\gamma_{it}(G)=\frac n 2$.
If (2) holds, the result follows by Theorem 3.2.
\end{proof}

By a long and difficult proof, Hansberg and Vlolkman \cite{Han} were able to characterize even
order trees $T$ with $\gamma(T)=\frac n 2-1$.
By Theorems 3.2, 3.3, and 3.5, the bipartite graphs $G$ with $\gamma_{it}(G)=|X|=|Y|=\frac n 2$ and $\gamma(G)=\frac n 2-1$
remain to be characterized. Some more effort must be used for completing this task.

\section {\large Independent transversal total domination }
A new variant of transversal in graphs was introduced very recently by  Mart\'{i}nez, Almira, and Yero
\cite {Mar}, called independent transversal total domination.
A dominating set $S$ of a graph $G$ is called a {\it total dominating set} is $G[S]$ has no
isolated vertices. The {\it total dominating number} of $G$, denoted by $\gamma_t(G)$,
is the cardinality of a minimum total dominating set of $G$. Further, a total dominating
set $S$ is called an {\it independent transversal total dominating set} if $S\cap I\neq \emptyset$ for each
$I\in \Omega(G)$. The {\it independent transversal total dominating number} of $G$, denoted by $\gamma_{tt}(G)$,
is the cardinality of a minimum independent transversal total dominating set of $G$.
So, for any graph $G$, $\gamma_{tt}(G)\geq \gamma_{it}(G)$.
Cockayne, Dawes and Hedetneimi \cite{Coc} obtained the following upper on the total domination number of a
connected graph in terms of the order of the graph.

\begin{theorem} (Cockayne, Dawes and Hedetneimi \cite{Coc})
If $G$ is a connected graph of order $n\geq 3$, then $\gamma_t(G)\leq \frac {2n} 3$.
\end{theorem}

Brigham, Carrington, and Yitary \cite{Brig} characterized the connected graphs of order at
least 3 with total domination number exactly two-thirds their order. Among other things,
Mart\'{i}nez, Almira, and Yero showed that $\gamma_{tt}(G)\leq \frac {2n} 3$ for some
classes of graphs of order $n$. Based on these results, they further asked the following open problem:

\noindent $\bullet$ \ Is it true that $\gamma_{tt}(G)\leq \frac {2n} 3$ for any graph of order $n$ ? If yes,
then: Is it true that $\gamma_{tt}(G)=\frac {2n} 3$ if and only if $\gamma_t(G)=\frac {2n} 3$ ?

\vspace{2mm} The answer for the question above is no.
By Corollary 2.3, for sufficiently large $n$, there exists a graph $G$ (the complement of a
triangle-free graph) of order $n$ with $$\gamma_{tt}(G)\geq \gamma_{it}(G)\geq n-9\sqrt{nlog\ n}.$$
If $\gamma_{tt}(G)\leq \frac {2n} 3$ holds for any graph $G$ of order $n$, we have $n-9\sqrt{nlog\ n}\leq \frac {2n} 3 $,
which is impossible for any sufficiently large $n$.



\end{document}